\documentclass[10pt]{amsart}
\usepackage[T1]{fontenc}
\usepackage[latin5]{inputenc}
\usepackage{graphics,epstopdf}
\usepackage[pdftex]{graphicx}
\usepackage{float}
\usepackage{amssymb}
\usepackage{amsfonts}
\usepackage[centertags]{amsmath}
\usepackage{amsthm}
\usepackage{newlfont}
 \makeatletter
\theoremstyle{plain}
\newtheorem{thm}{Theorem}[section]
\numberwithin{equation}{section}
\numberwithin{figure}{section}  
\theoremstyle{plain}

\theoremstyle{plain}
\newtheorem{cor}[thm]{Corollary} 
\theoremstyle{plain}

\theoremstyle{plain}
\newtheorem{lem}[thm]{Lemma} 
\theoremstyle{plain}

\begin{document}
\title{Approximation by a Kantorovich type $q$-Bernstein-Stancu Operators}
\author{M. Mursaleen*, Khursheed J. Ansari and Asif Khan}
\thanks{*Corresponding author.}
\address{Department of Mathematics, Aligarh Muslim University,
Aligarh 202002, India}
\email{mursaleenm@gmail.com; ansari.jkhursheed@gmail.com; asifjnu07@gmail.com}


\begin{abstract}
In this paper, we introduce a Kantorovich type generalization of $q$%
-Bernstein-Stancu operators. We study the convergence of the introduced
operators and also obtain the rate of convergence by these operators in
terms of the modulus of continuity. Further, we study local approximation
property and Voronovskaja type theorem for the said operators. We show
comparisons and some illustrative graphics for the convergence of operators
to a certain function.
\end{abstract}

\maketitle

\section{Introduction}

\parindent=8mmThe applications of $q$-calculus in the area of approximation
theory were initiated by Lupas \cite{lup}, who first introduced $q$%
-Bernstein polynomials. Later, Phillips \cite{phi} proposed other $q$-analog
of Bernstein polynomials which became very popular and several researchers
obtained the interesting approximation properties for $q$-Bernstein
polynomials. In the recent years, many researchers have studied the
approximation properties for linear positive operators \cite{pn, aral, gup,
gup1, vish, moh, Moh1, ma511, mkeg}. Mursaleen et al in \cite{mfa1, mfa2,
mksn, ma1} have also obtained statistical approximation properties for new
positive linear operators and some approximation theorems for generalized $q$%
-Bernstein-Schurer operators.

\parindent=8mm Initially, we start of with the basic definitions and
notations of quantum calculus \cite{kac}. Let $q>0$ be a fixed real number.
For any $n\in\mathbb{N}\cup\{0\}$, the $q$-integer $[n]=[n]_q$ is defined by

\begin{equation*}
\lbrack n]:=\left\{
\begin{array}{l}
\frac{(1-q^{n})}{(1-q)},~~~~~~q\neq 1 \\
~n,~~~~~~~~~~~q=1%
\end{array}%
\right.
\end{equation*}%
\newline
and $q$-factorial $[n]!=[n]_{q}!$ by
\begin{equation*}
\lbrack n]!:=\left\{
\begin{array}{l}
\lbrack n][n-1]\cdots \lbrack 1],~~~~~~~n\geq 1 \\
~1,~~~~~~~~~~~~~~~~~~~~~~~~~n=0.%
\end{array}%
\right.
\end{equation*}%
For integers $0\leqslant k\leqslant n$, the $q$-binomial coefficient is
defined by
\begin{equation*}
\left[
\begin{array}{c}
n \\
k%
\end{array}%
\right] :=\frac{[n]!}{[k]![n-k]!}.
\end{equation*}%
The $q$-analogue of integration in the interval $[0,A]$ is defined by
\begin{equation*}
\int_{0}^{A}f(t)d_{q}t:=A(1-q)\sum\limits_{n=0}^{\infty
}f(Aq^{n})q^{n},~~0<q<1.
\end{equation*}

\parindent=8mm In \cite{ber}, Bernstein introduced the following well-known
positive linear operators
\begin{equation}  \label{1.1}
B_n(f;x)=\sum\limits_{k=0}^{n}f\left(\frac kn\right){\binom{n}{k}}%
x^k(1-x)^{n-k}
\end{equation}
and he showed that if $f\in C[0,1]$, then $B_n(f;x)\rightrightarrows f(x)$
where $"\rightrightarrows"$ represents the uniform convergence. One can find
more details about the Bernstein polynomials in \cite{lor}. The $q$%
-generalization of the Bernstein polynomials was introduced by G.M. Philips
\cite{phi}.

\parindent=8mm The classical Kantorovich operator $B_n^*,n=1,2,\cdots$ is
defined by (cf. \cite{lor})
\begin{align*}
B_n^*(f;x)&=(n+1)\sum\limits_{k=0}^{n}{\binom{n}{k}}x^k(1-x)^{n-k}\int^{%
\frac{k+1}{n+1}}_{\frac{k}{n+1}} f(t)dt
\end{align*}
\begin{align}  \label{1.2}
~~~~~~~=\sum\limits_{k=0}^{n}{\binom{n}{k}}x^k(1-x)^{n-k}\int^{1}_{0} f%
\left(\frac{k+t}{n+1}\right)dt
\end{align}

\parindent=8mm Inspired by (1.2), Mahmudov \cite{mah} introduced a $q$-type
generalization of Bernstein-Kantorovich operators as follows:
\begin{equation}
B_{n,q}^{\ast }(f;x)=\sum\limits_{k=0}^{n}p_{n,k}(q;x)\int_{0}^{1}f\biggl{(}%
\frac{[k]+q^{k}t}{[n+1]}\biggl{)}d_{q}t  \label{1.3}
\end{equation}%
where
\begin{equation*}
p_{n,k}(q,x):=\left[
\begin{array}{c}
n \\
k%
\end{array}%
\right] x^{k}(1-x)_{q}^{n-k},~~(1-x)_{q}^{n}=\prod\limits_{s=0}^{n-1}\big(%
1-q^{s}x\big).
\end{equation*}%
It can be seen that for $q\rightarrow 1^{-}$ the $q$-Bernstein-Kantorovich
operator becomes the classical Bernstein-Kantorovich operator (1.2).

\parindent=8mm \parindent=8mm In 2010, Gadjiev and Gorbanalizadeh \cite{gad}
introduced the following construction of Bernstein-Stancu type polynomials
with shifted knots:

\begin{equation}  \label{1.4}
S_{n,\alpha,\beta}(f;x)=\biggl{(}\frac{n+\beta_2}{n}\biggl{)}%
^n\sum\limits_{r=0}^{n}f\biggl{(}\frac{r+\alpha_1}{n+\beta_1}\biggl{)}{%
\binom{n}{r}}\biggl{(}x-\frac{\alpha_2}{n+\beta_2}\biggl{)}^r\biggl{(}\frac{%
n+\alpha_2}{n+\beta_2}-x\biggl{)}^{n-r}
\end{equation}
where $\frac{\alpha_2}{n+\beta_2}\leqslant x\leqslant\frac{n+\alpha_2}{%
n+\beta_2}$ and $\alpha_k,\beta_k~(k=1,2)$ are positive real numbers
provided $0\leqslant\alpha_1\leqslant\alpha_2\leqslant\beta_1\leqslant\beta_2
$. It is clear that for $\alpha_2=\beta_2=0$, then polynomials (1.4) turn
into the Bernstein-Stancu polynomials (1.2) and if $\alpha_1=\alpha_2=%
\beta_1=\beta_2=0$ then these polynomials turn into the classical Bernstein
polynomials.

\parindent=8mm Motivated by (1.4), İçöz \cite{gur} introduced a Kantorovich
type generalization of Bernstein-Stancu polynomials as follows:
\begin{equation}  \label{1.5}
S_{n,\alpha,\beta}^*(f;x)=(n+\beta_1+1)\biggl{(}\frac{n+\beta_2}{n}\biggl{)}%
^n~\sum\limits_{r=0}^{n}{\binom{n}{r}}\biggl{(}x-\frac{\alpha_2}{n+\beta_2}%
\biggl{)}^r\biggl{(}\frac{n+\alpha_2}{n+\beta_2}-x\biggl{)}^{n-r} \int^{%
\frac{r+\alpha_1+1}{n+\beta_1+1}}_{\frac{r+\alpha_1}{n+\beta_1+1}} f(s)ds.
\end{equation}
where $\frac{\alpha_2}{n+\beta_2}\leqslant x\leqslant\frac{n+\alpha_2}{%
n+\beta_2}$ and $\alpha_k,\beta_k~(k=1,2)$ are positive real numbers
provided $0\leqslant\alpha_1\leqslant\alpha_2\leqslant\beta_1\leqslant\beta_2
$. It is clear that for $\alpha_1=\alpha_2=\beta_1=\beta_2=0$  then these
polynomials turn into the Bernstein-Kantorovich operators.

\section{Construction of new operators and some auxiliary results}

We construct a Kantorovich type $q$-Bernstein-Stancu type polynomials as
follows:
\begin{equation}
K_{n,q}^{(\alpha ,\beta )}=\biggl{(}\frac{[n]+\beta _{2}}{[n]}\biggl{)}%
^{n}~\sum\limits_{k=0}^{n}\left[
\begin{array}{c}
n \\
k%
\end{array}%
\right] \biggl{(}x-\frac{\alpha _{2}}{[n]+\beta _{2}}\biggl{)}_{q}^{k}~%
\biggl{(}\frac{[n]+\alpha _{2}}{[n]+\beta _{2}}-x\biggl{)}%
_{q}^{n-k}\int_{0}^{1}f\biggl{(}\frac{[k]+q^{k}t+\alpha _{1}}{[n+1]+\beta
_{1}}\biggl{)}d_{q}t  \label{2.1}
\end{equation}%
where $\frac{\alpha _{2}}{n+\beta _{2}}\leqslant x\leqslant \frac{n+\alpha
_{2}}{n+\beta _{2}}$ and $\alpha _{k},\beta _{k}~(k=1,2)$ are positive real
numbers provided $0\leqslant \alpha _{1}\leqslant \alpha _{2}\leqslant \beta
_{1}\leqslant \beta _{2}$. If we put for $\alpha _{1}=\alpha _{2}=\beta
_{1}=\beta _{2}=0$ in (2.1) then these polynomials turn into the
Bernstein-Kantorovich operators (1.3) introduced by Mahmudov. \parindent=8mm
Throughout this paper, $\Vert .\Vert $ denotes the sup-norm on [0,1].
Further, $C$ denotes the absolutely positive constant not necessarily the
same at each occurrence.

\parindent=8mm The aim of this paper is to study some approximation
properties of Kantorovich type $q$-Bernstein-Stancu operators defined by
(2.1). First, we prove the basic convergence of the introduced operators and
also obtain the rate of convergence by these operators in terms of the
modulus of continuity. Further, we study local approximation property and
Voronovskaja type theorem for the said operators. With the help of the
Matlab we show comparisons and some illustrative graphics for the
convergence of operators to a function. \newline

\begin{lem}
\label{2.1} Let $K^{(\alpha,\beta)}_{n,q}(f;x)$ be given by (2.1). Then the
following properties hold:\newline
(i)
\begin{equation*}
K^{(\alpha,\beta)}_{n,q}(1;x)=1;
\end{equation*}
(ii)
\begin{eqnarray}
K^{(\alpha,\beta)}_{n,q}(t;x)&=& \frac{[n]+\beta_2}{[n+1]+\beta_1}\frac{2q}{%
[2]}\left(x-\frac{\alpha_2}{[n]+\beta_2}\right)+\frac1{[n+1] +\beta_1}%
\left(\alpha_1+\frac1{[2]}\right)  \notag
\end{eqnarray}
(iii)
\begin{eqnarray}
K^{(\alpha,\beta)}_{n,q}(t^2;x)&=&\frac{q[n-1]}{[n]}\biggl{(}1+\frac{(q-1)^2%
}{[3]}+\frac{2(q-1)}{[2]}\biggl{)}\biggl{(}\frac{[n] +\beta_2}{[n+1]+\beta_1}%
\biggl{)}^2\biggl{(}x-\frac{\alpha_2}{[n]+\beta_2}\biggl{)}^2_q  \notag \\
&&+ \biggl{(}1+\frac{q^2-1}{[3]}+(2\alpha_1+1)\frac{2q}{1+q}\biggl{)} %
\biggl{(}\frac{[n]+\beta_2}{([n+1]+\beta_1)^2}\biggl{)}\biggl{(}x-\frac{%
\alpha_2}{[n]+\beta_2}\biggl{)}  \notag \\
&&+\frac1{([n+1]+\beta_1)^2}\biggl{(}\alpha_1^2+\frac{2\alpha_1}{[2]}%
+\frac1{[3]}\biggl{)}  \notag
\end{eqnarray}%
\newline
\end{lem}

\begin{proof}
(i) Using binomial coefficient
\begin{eqnarray*}
K_{n,q}^{(\alpha ,\beta )}(1;x) &=&\biggl{(}\frac{[n]+\beta _{2}}{[n]}%
\biggl{)}^{n}~\sum\limits_{k=0}^{n}\left[
\begin{array}{c}
n \\
k%
\end{array}%
\right] \biggl{(}x-\frac{\alpha _{2}}{[n]+\beta _{2}}\biggl{)}_{q}^{k}~%
\biggl{(}\frac{[n]+\alpha _{2}}{[n]+\beta _{2}}-x\biggl{)}%
_{q}^{n-k}\int_{0}^{1}d_{q}t \\
&=&\biggl{(}\frac{[n]+\beta _{2}}{[n]}\biggl{)}^{n}~\sum\limits_{k=0}^{n}%
\left[
\begin{array}{c}
n \\
k%
\end{array}%
\right] \biggl{(}x-\frac{\alpha _{2}}{[n]+\beta _{2}}\biggl{)}_{q}^{k}~%
\biggl{(}\frac{[n]+\alpha _{2}}{[n]+\beta _{2}}-x\biggl{)}_{q}^{n-k}=1.
\end{eqnarray*}

(ii) For $f(t)=t$, we have
\begin{equation*}
K_{n,q}^{(\alpha ,\beta )}=\biggl{(}\frac{[n]+\beta _{2}}{[n]}\biggl{)}%
^{n}~\sum\limits_{k=0}^{n}\left[
\begin{array}{c}
n \\
k%
\end{array}%
\right] \biggl{(}x-\frac{\alpha _{2}}{[n]+\beta _{2}}\biggl{)}_{q}^{k}~%
\biggl{(}\frac{[n]+\alpha _{2}}{[n]+\beta _{2}}-x\biggl{)}%
_{q}^{n-k}\int_{0}^{1}\frac{[k]+\alpha _{1}+q^{k}t}{[n+1]+\beta _{1}}d_{q}t
\end{equation*}%
\begin{equation*}
K_{n,q}^{(\alpha ,\beta )}=\biggl{(}\frac{[n]+\beta _{2}}{[n]}\biggl{)}%
^{n}~\sum\limits_{k=0}^{n}\left[
\begin{array}{c}
n \\
k%
\end{array}%
\right] \biggl{(}x-\frac{\alpha _{2}}{[n]+\beta _{2}}\biggl{)}_{q}^{k}~%
\biggl{(}\frac{[n]+\alpha _{2}}{[n]+\beta _{2}}-x\biggl{)}_{q}^{n-k}%
\biggl{\{}\frac{[k]+\alpha _{1}}{[n+1]+\beta _{1}}+\frac{q^{k}}{%
[2]([n+1]+\beta _{1})}\biggl{\}}
\end{equation*}%
\begin{eqnarray*}
K_{n,q}^{(\alpha ,\beta )} &=&\frac{1}{[n+1]+\beta _{1}}\biggl{(}\frac{%
[n]+\beta _{2}}{[n]}\biggl{)}^{n}~\sum\limits_{k=0}^{n}\left[
\begin{array}{c}
n \\
k%
\end{array}%
\right] \biggl{(}x-\frac{\alpha _{2}}{[n]+\beta _{2}}\biggl{)}_{q}^{k}~%
\biggl{(}\frac{[n]+\alpha _{2}}{[n]+\beta _{2}}-x\biggl{)}_{q}^{n-k}%
\biggl{\{}\frac{2q}{[2]}[k]+\alpha _{1}+\frac{1}{[2]}\biggl{\}} \\
(\text{using}~q^{k} &=&1+(q-1)[k])
\end{eqnarray*}%
\begin{eqnarray*}
K_{n,q}^{(\alpha ,\beta )} &=&\frac{1}{[n+1]+\beta _{1}}\biggl{\{}\frac{2q}{%
[2]}\biggl{(}\frac{[n]+\beta _{2}}{[n]}\biggl{)}^{n}~\sum\limits_{k=1}^{n}%
\frac{[n]}{[k]}\left[
\begin{array}{c}
n-1 \\
k-1%
\end{array}%
\right] \biggl{(}x-\frac{\alpha _{2}}{[n]+\beta _{2}}\biggl{)}_{q}^{k}~%
\biggl{(}\frac{[n]+\alpha _{2}}{[n]+\beta _{2}}-x\biggl{)}_{q}^{n-k}[k] \\
&&+\biggl{(}\alpha _{1}+\frac{1}{[2]}\biggl{)}\biggl{(}\frac{[n]+\beta _{2}}{%
[n]}\biggl{)}^{n}~\sum\limits_{k=0}^{n}\left[
\begin{array}{c}
n \\
k%
\end{array}%
\right] \biggl{(}x-\frac{\alpha _{2}}{[n]+\beta _{2}}\biggl{)}_{q}^{k}~%
\biggl{(}\frac{[n]+\alpha _{2}}{[n]+\beta _{2}}-x\biggl{)}_{q}^{n-k}%
\biggl{\}} \\
&=&\frac{[n]}{[n+1]+\beta _{1}}\frac{2q}{[2]}\biggl{(}\frac{[n]+\beta _{2}}{%
[n]}\biggl{)}^{n}~\sum\limits_{k=0}^{n-1}\left[
\begin{array}{c}
n-1 \\
k%
\end{array}%
\right] \biggl{(}x-\frac{\alpha _{2}}{[n]+\beta _{2}}\biggl{)}_{q}^{k+1}~%
\biggl{(}\frac{[n]+\alpha _{2}}{[n]+\beta _{2}}-x\biggl{)}_{q}^{n-k-1}\frac{1%
}{[n+1]+\beta _{1}}\biggl{(}\alpha _{1}+\frac{1}{[2]}\biggl{)} \\
&=&\frac{[n]+\beta _{2}}{[n+1]+\beta _{1}}\frac{2q}{[2]}\biggl{(}x-\frac{%
\alpha _{2}}{[n]+\beta _{2}}\biggl{)}+\frac{1}{[n+1]+\beta _{1}}\biggl{(}%
\alpha _{1}+\frac{1}{[2]}\biggl{)}
\end{eqnarray*}

(iii) For $f(t)=t^{2}$, in view of (2.1), we have
\begin{eqnarray*}
K_{n,q}^{(\alpha ,\beta )} &=&\biggl{(}\frac{[n]+\beta _{2}}{[n]}\biggl{)}%
^{n}~\sum\limits_{k=0}^{n}\left[
\begin{array}{c}
n \\
k%
\end{array}%
\right] \biggl{(}x-\frac{\alpha _{2}}{[n]+\beta _{2}}\biggl{)}_{q}^{k}~%
\biggl{(}\frac{[n]+\alpha _{2}}{[n]+\beta _{2}}-x\biggl{)}%
_{q}^{n-k}\int_{0}^{1}\biggl{(}\frac{[k]+q^{k}t+\alpha _{1}}{[n+1]+\beta _{1}%
}\biggl{)}^{2}d_{q}t \\
&=&\frac{1}{([n+1]+\beta _{1})^{2}}\biggl{(}\frac{[n]+\beta _{2}}{[n]}%
\biggl{)}^{n}~\sum\limits_{k=0}^{n}\left[
\begin{array}{c}
n \\
k%
\end{array}%
\right] \biggl{(}x-\frac{\alpha _{2}}{[n]+\beta _{2}}\biggl{)}_{q}^{k}~%
\biggl{(}\frac{[n]+\alpha _{2}}{[n]+\beta _{2}}-x\biggl{)}%
_{q}^{n-k}\int_{0}^{1}\bigl{(}\lbrack k]+(1+(q-1)[k])t+\alpha _{1}\bigl{)}%
^{2}d_{q}t \\
&=&\frac{1}{([n+1]+\beta _{1})^{2}}\biggl{(}\frac{[n]+\beta _{2}}{[n]}%
\biggl{)}^{n}~\sum\limits_{k=0}^{n}\left[
\begin{array}{c}
n \\
k%
\end{array}%
\right] \biggl{(}x-\frac{\alpha _{2}}{[n]+\beta _{2}}\biggl{)}_{q}^{k}%
\biggl{(}\frac{[n]+\alpha _{2}}{[n]+\beta _{2}}-x\biggl{)}_{q}^{n-k}%
\biggl{\{}\biggl{(}1+\frac{(q-1)^{2}}{[3]}+\frac{2(q-1)}{[2]}\biggl{)}%
\lbrack k]^{2} \\
&&+\biggl{(}\frac{2(q-1)}{[3]}+2\alpha _{1}+\frac{2\alpha _{1}(q-1)}{[2]}+%
\frac{2}{[2]}\biggl{)}\lbrack k]+\alpha _{1}^{2}+\frac{2\alpha _{1}}{[2]}+%
\frac{1}{[3]}\biggl{\}} \\
&=&\frac{1}{([n+1]+\beta _{1})^{2}}\biggl{\{}\biggl{(}1+\frac{(q-1)^{2}}{[3]}%
+\frac{2(q-1)}{[2]}\biggl{)}\biggl{(}\frac{[n]+\beta _{2}}{[n]}\biggl{)}%
^{n}\sum\limits_{k=1}^{n}\frac{[n]}{[k]}\left[
\begin{array}{c}
n-1 \\
k-1%
\end{array}%
\right] \biggl{(}x-\frac{\alpha _{2}}{[n]+\beta _{2}}\biggl{)}_{q}^{k}~%
\biggl{(}\frac{[n]+\alpha _{2}}{[n]+\beta _{2}}-x\biggl{)}_{q}^{n-k}[k]^{2}
\\
&&+\biggl{(}\frac{2(q-1)}{[3]}+2\alpha _{1}+\frac{2\alpha _{1}(q-1)}{[2]}+%
\frac{2}{[2]}\biggl{)}\biggl{(}\frac{[n]+\beta _{2}}{[n]}\biggl{)}%
^{n}\sum\limits_{k=0}^{n}\left[
\begin{array}{c}
n \\
k%
\end{array}%
\right] \biggl{(}x-\frac{\alpha _{2}}{[n]+\beta _{2}}\biggl{)}_{q}^{k}~%
\biggl{(}\frac{[n]+\alpha _{2}}{[n]+\beta _{2}}-x\biggl{)}_{q}^{n-k}[k] \\
&&+\biggl{(}\alpha _{1}^{2}+\frac{2\alpha _{1}}{[2]}+\frac{1}{[3]}\biggl{)}%
\biggl{(}\frac{[n]+\beta _{2}}{[n]}\biggl{)}^{n}~\sum\limits_{k=0}^{n}\left[
\begin{array}{c}
n \\
k%
\end{array}%
\right] \biggl{(}x-\frac{\alpha _{2}}{[n]+\beta _{2}}\biggl{)}_{q}^{k}~%
\biggl{(}\frac{[n]+\alpha _{2}}{[n]+\beta _{2}}-x\biggl{)}_{q}^{n-k}%
\biggl{\}} \\
&=&\frac{1}{([n+1]+\beta _{1})^{2}}\biggl{\{}\biggl{(}1+\frac{(q-1)^{2}}{[3]}%
+\frac{2(q-1)}{[2]}\biggl{)}\biggl{(}\frac{[n]+\beta _{2}}{[n]}\biggl{)}^{n}
\\
&&\times \sum\limits_{k=0}^{n-1}[n]\left[
\begin{array}{c}
n-1 \\
k%
\end{array}%
\right] \biggl{(}x-\frac{\alpha _{2}}{[n]+\beta _{2}}\biggl{)}_{q}^{k+1}~%
\biggl{(}\frac{[n]+\alpha _{2}}{[n]+\beta _{2}}-x\biggl{)}_{q}^{n-k-1}[k+1]
\\
&&+\biggl{(}\frac{2(q-1)}{[3]}+2\alpha _{1}+\frac{2\alpha _{1}(q-1)}{[2]}+%
\frac{2}{[2]}\biggl{)}\biggl{(}\frac{[n]+\beta _{2}}{[n]}\biggl{)}^{n} \\
&&\times \sum\limits_{k=0}^{n-1}[n]\left[
\begin{array}{c}
n-1 \\
k%
\end{array}%
\right] \biggl{(}x-\frac{\alpha _{2}}{[n]+\beta _{2}}\biggl{)}_{q}^{k+1}~%
\biggl{(}\frac{[n]+\alpha _{2}}{[n]+\beta _{2}}-x\biggl{)}_{q}^{n-k-1}+%
\biggl{(}\alpha _{1}^{2}+\frac{2\alpha _{1}}{[2]}+\frac{1}{[3]}\biggl{)}%
\biggl{\}} \\
&=&\frac{1}{([n+1]+\beta _{1})^{2}}\biggl{\{}\biggl{(}1+\frac{(q-1)^{2}}{[3]}%
+\frac{2(q-1)}{[2]}\biggl{)}\biggl{(}\frac{[n]+\beta _{2}}{[n]}\biggl{)}^{n}
\\
&&\times \sum\limits_{k=1}^{n-1}[n]\frac{[n-1]}{[k]}\left[
\begin{array}{c}
n-2 \\
k-1%
\end{array}%
\right] \biggl{(}x-\frac{\alpha _{2}}{[n]+\beta _{2}}\biggl{)}_{q}^{k+1}~%
\biggl{(}\frac{[n]+\alpha _{2}}{[n]+\beta _{2}}-x\biggl{)}_{q}^{n-k-1}q[k] \\
&&+\biggl{(}1+\frac{(q-1)^{2}}{[3]}+\frac{2(q-1)}{[2]}+\frac{2(q-1)}{[3]}%
+2\alpha _{1}+\frac{2\alpha _{1}(q-1)}{[2]}+\frac{2}{[2]}\biggl{)}\biggl{(}%
\frac{[n]+\beta _{2}}{[n]}\biggl{)}^{n} \\
&&\times \sum\limits_{k=0}^{n-1}[n]\left[
\begin{array}{c}
n-1 \\
k%
\end{array}%
\right] \biggl{(}x-\frac{\alpha _{2}}{[n]+\beta _{2}}\biggl{)}_{q}^{k+1}~%
\biggl{(}\frac{[n]+\alpha _{2}}{[n]+\beta _{2}}-x\biggl{)}_{q}^{n-k-1}+%
\biggl{(}\alpha _{1}^{2}+\frac{2\alpha _{1}}{[2]}+\frac{1}{[3]}\biggl{)}%
\biggl{\}} \\
&=&\frac{1}{([n+1]+\beta _{1})^{2}}\biggl{\{}\biggl{(}1+\frac{(q-1)^{2}}{[3]}%
+\frac{2(q-1)}{[2]}\biggl{)}\biggl{(}\frac{[n]+\beta _{2}}{[n]}\biggl{)}^{n}
\\
&&\times \sum\limits_{k=0}^{n-2}q[n][n-1]\left[
\begin{array}{c}
n-2 \\
k%
\end{array}%
\right] \biggl{(}x-\frac{\alpha _{2}}{[n]+\beta _{2}}\biggl{)}_{q}^{k+2}~%
\biggl{(}\frac{[n]+\alpha _{2}}{[n]+\beta _{2}}-x\biggl{)}_{q}^{n-k-2} \\
&&+\biggl{(}1+\frac{q^{2}-1}{[3]}+(2\alpha _{1}+1)\frac{2q}{[2]}\biggl{)}%
\biggl{(}\frac{[n]+\beta _{2}}{[n]}\biggl{)}^{n}~\sum\limits_{k=0}^{n-1}[n]%
\left[
\begin{array}{c}
n-1 \\
k%
\end{array}%
\right] \biggl{(}x-\frac{\alpha _{2}}{[n]+\beta _{2}}\biggl{)}_{q}^{k+1}%
\biggl{(}\frac{[n]+\alpha _{2}}{[n]+\beta _{2}}-x\biggl{)}_{q}^{n-k-1} \\
&&+\biggl{(}\alpha _{1}^{2}+\frac{2\alpha _{1}}{[2]}+\frac{1}{[3]}\biggl{)}%
\biggl{\}} \\
&=&\frac{q[n-1]}{[n]}\biggl{(}1+\frac{(q-1)^{2}}{[3]}+\frac{2(q-1)}{[2]}%
\biggl{)}\biggl{(}\frac{[n]+\beta _{2}}{[n+1]+\beta _{1}}\biggl{)}^{2}%
\biggl{(}x-\frac{\alpha _{2}}{[n]+\beta _{2}}\biggl{)}_{q}^{2} \\
&&+\biggl{(}1+\frac{q^{2}-1}{[3]}+(2\alpha _{1}+1)\frac{2q}{[2]}\biggl{)}%
\biggl{(}\frac{[n]+\beta _{2}}{([n+1]+\beta _{1})^{2}}\biggl{)}\biggl{(}x-%
\frac{\alpha _{2}}{[n]+\beta _{2}}\biggl{)}\frac{1}{([n+1]+\beta _{1})^{2}}%
\biggl{(}\alpha _{1}^{2}+\frac{2\alpha _{1}}{[2]}+\frac{1}{[3]}\biggl{)}
\end{eqnarray*}
\end{proof}

\begin{lem}
\label{2.2} For all $x\in\left[\frac{\alpha_2}{[n]+\beta_2},\frac{%
[n]+\alpha_2}{[n]+\beta_2}\right]$, we have
\begin{eqnarray*}
  K^{(\alpha,\beta)}_{n,q}\bigl{(}(t-x)^2;x\bigl{)}&\leq& \frac{2q^2(2q+1)}{[2][3]}\frac{[n]([n]+\alpha_2)}{([n+1]+\beta_1)^2}
   +\frac{q}{1+q}\left(\frac{3+5q+4q^2}{1+q+q^2}+4\alpha_1\right)\frac{[n]}{([n+1]+\beta_1)^2}\\
  &&-\frac{2}{1+q}\frac{(2q[n]+2\alpha_1+1)([n]+\alpha_2)}{([n+1]+\beta_1)([n]+\beta_2)}
  +\left(\frac{[n]+\alpha_2}{[n]+\beta_2}\right)^2+\left(\frac{1+\alpha_1}{[n+1]+\beta_1}\right)^2
\end{eqnarray*}
\end{lem}

\begin{proof}
From Lemma 2.1, we have
\begin{eqnarray*}
K^{(\alpha,\beta)}_{n,q}\bigl{(}(t-x)^2;x\bigl{)} &=& \frac{q[n-1]}{[n]}%
\biggl{(}1+\frac{(q-1)^2}{[3]}+\frac{2(q-1)}{[2]}\biggl{)}\biggl{(}\frac{%
[n]+\beta_2}{[n+1]+\beta_1}\biggl{)}^2 \biggl{(}x-\frac{\alpha_2}{[n]+\beta_2%
}\biggl{)}^2_q \\
&&+\biggl{(}1+\frac{q^2-1}{[3]}+(2\alpha_1+1)\frac{2q}{1+q}\biggl{)} %
\biggl{(}\frac{[n]+\beta_2}{([n+1]+\beta_1)^2}\biggl{)} \biggl{(}x-\frac{%
\alpha_2}{[n]+\beta_2}\biggl{)} \\
&&+\frac1{([n+1]+\beta_1)^2}\biggl{(}\alpha_1^2+\frac{2\alpha_1}{[2]}%
+\frac1{[3]}\biggl{)} \\
&& -2x\biggl{\{}\frac{[n]+\beta_2}{[n+1]+\beta_1}\frac{2q}{[2]}\biggl{(}x-%
\frac{\alpha_2}{[n]+\beta_2}\biggl{)}+\frac1{[n+1]+\beta_1} \biggl{(}%
\alpha_1+\frac1{[2]}\biggl{)}\biggl{\}}+x^2
\end{eqnarray*}

\begin{eqnarray*}
K^{(\alpha,\beta)}_{n,q}\bigl{(}(t-x)^2;x\bigl{)} &=& \biggl{\{}\frac{q[n-1]%
}{[n]}\biggl{(}1+\frac{(q-1)^2}{[3]}+\frac{2(q-1)}{[2]}\biggl{)}\biggl{(}%
\frac{[n]+\beta_2}{[n+1]+\beta_1}\biggl{)}^2 -\frac{4q}{1+q}\frac{[n]+\beta_2%
}{[n+1]+\beta_1}+1\biggl{\}}x^2 \\
&& +\biggl{\{}\biggl{(}1+\frac{q^2-1}{[3]}+(2\alpha_1+1)\frac{2q}{1+q}%
\biggl{)}\frac{[n]+\beta_2}{([n+1]+\beta_1)^2} \\
&& -\frac{q[n-1]}{[n]}\frac{[2]\alpha_2}{[n]+\beta_2}\biggl{(}1+\frac{(q-1)^2%
}{[3]} +\frac{2(q-1)}{[2]}\biggl{)}\biggl{(}\frac{[n]+\beta_2}{[n+1]+\beta_1}%
\biggl{)}^2 \\
&& +\frac{2q}{1+q}\frac{2\alpha_2}{[n]+\beta_2}\biggl{(}\frac{[n]+\beta_2}{%
[n+1]+\beta_1}\biggl{)} -\frac2{[n+1]+\beta_1}\biggl{(}\alpha_1+\frac1{[2]}%
\biggl{)}\biggl{\}}x \\
&&+\frac{q^2[n-1]}{[n]}\biggl{(}1+\frac{(q-1)^2}{[3]}+\frac{2(q-1)}{[2]}%
\biggl{)}\biggl{(}\frac{\alpha_2}{[n+1]+\beta_1}\biggl{)}^2 \\
&&-\biggl{(}1+\frac{q^2-1}{[3]}+(2\alpha_1+1)\frac{2q}{1+q}\biggl{)}\frac{%
\alpha_2}{([n+1]+\beta_1)^2} +\frac1{([n+1]+\beta_1)^2}\biggl{(}\alpha_1^2+%
\frac{2\alpha_1}{[2]}+\frac1{[3]}\biggl{)}
\end{eqnarray*}
By using the monotonicity of positive linear operators $K^{(\alpha,\beta)}_{n,q}$ over $\left[\frac{\alpha_2}{[n]+\beta_2},\frac{%
[n]+\alpha_2}{[n]+\beta_2}\right]$, we have
\begin{eqnarray*}
   K^{(\alpha,\beta)}_{n,q}\bigl{(}(t-x)^2;x\bigl{)} &\leq& \left(1-\frac1{[n]}\right)\biggl{(}1+\frac{(q-1)^2}{[3]}+\frac{2(q-1)}{[2]}\biggl{)}
   \biggl{\{}\left(\frac{[n]+\alpha_2}{[n+1]+\beta_1}\right)^2-\frac{[2]\alpha_2([n]+\alpha_2)}{([n+1]+\beta_1)^2}+\frac{q\alpha_2^2}{([n+1]+\beta_1)^2}\biggl{\}} \\
  &&+ \biggl{(}1+\frac{q^2-1}{[3]}+(2\alpha_1+1)\frac{2q}{1+q}\biggl{)}\frac{[n]}{([n+1]+\beta_1)^2}+\left(\frac{[n]+\alpha_2}{[n]+\beta_2}\right)^2 \\
   && +\frac{4q}{1+q}\frac{\alpha_2([n]+\alpha_2)}{([n+1]+\beta_1)([n]+\beta_2)}-2\left(\alpha_1+\frac1{[2]}\right)\frac{([n]+\alpha_2)}{([n+1]+\beta_1)([n]+\beta_2)}\\
   &&-\frac{4q}{1+q}\frac{([n]+\alpha_2)^2}{([n+1]+\beta_1)([n]+\beta_2)}+\frac1{([n+1]+\beta_1)^2}\biggl{(}\alpha_1^2+%
\frac{2\alpha_1}{[2]}+\frac1{[3]}\biggl{)}
\end{eqnarray*}
\begin{eqnarray*}
   &=&  \left(1-\frac1{[n]}\right)\frac{2q^2(2q+1)}{[2][3]}\frac{[n]^2+(1-q)[n]\alpha_2}{([n+1]+\beta_1)^2}
   +\frac{q}{1+q}\left(\frac{3+5q+4q^2}{1+q+q^2}+4\alpha_1\right)\frac{[n]}{([n+1]+\beta_1)^2}+\left(\frac{[n]+\alpha_2}{[n]+\beta_2}\right)^2 \\
   &&+\frac{2(2q\alpha_2-[2]\alpha_1-1)}{[2]}\frac{[n]+\alpha_2}{([n+1]+\beta_1)([n]+\beta_2)} -\frac{4q}{1+q}\frac{([n]+\alpha_2)^2}{([n+1]+\beta_1)([n]+\beta_2)}+\frac1{([n+1]+\beta_1)^2}\biggl{(}\alpha_1^2+
\frac{2\alpha_1}{[2]}+\frac1{[3]}\biggl{)}
\end{eqnarray*}
\begin{eqnarray*}
   &\leq&  \frac{2q^2(2q+1)}{[2][3]}\frac{[n]([n]+\alpha_2)}{([n+1]+\beta_1)^2}
   +\frac{q}{1+q}\left(\frac{3+5q+4q^2}{1+q+q^2}+4\alpha_1\right)\frac{[n]}{([n+1]+\beta_1)^2} \\
   &&-\frac2{1+q}\frac{(2q[n]+2\alpha_1+1)([n]+\alpha_2)}{([n+1]+\beta_1)([n]+\beta_2)} +\left(\frac{[n]+\alpha_2}{[n]+\beta_2}\right)^2+\left(\frac{1+\alpha_1}{[n+1]+\beta_1}\right)^2
\end{eqnarray*}
which is the required result.
\end{proof}

\begin{lem}
\label{2.3} Assume that $0<q_n<1$, $q_n\to1$ and $q_n^n\to a~(0\leq a<1)$ as $n\to\infty$%
. Then we have
\begin{equation*}
\lim\limits_{n\to\infty}[n]_{q_n}K^{(\alpha,\beta)}_{n,q_n}\bigl{(}t-x;x%
\bigl{)}=-\frac{1+a+2(\beta_1-\beta_2)}{2}x+\frac{1+2(\alpha_1-\alpha_2)}{2};
\end{equation*}
\begin{equation*}
\lim\limits_{n\to\infty}[n]_{q_n}K^{(\alpha,\beta)}_{n,q_n}\bigl{(}(t-x)^2;x%
\bigl{)}=\bigl{(}a+2\beta_1-2\beta_2\bigl{)}x^2+x.
\end{equation*}

\begin{proof}
To prove the lemma we use formulae for $K_{n,q_n}^{\alpha,\beta}(t;x)$ and $%
K_{n,q_n}^{\alpha,\beta}(t^2;x)$ given in Lemma 2.1.\newline
\begin{eqnarray*}
\lim\limits_{n\to\infty}[n]_{q_n}K^{(\alpha,\beta)}_{n,q_n}\bigl{(}t-x;x%
\bigl{)} &=& \lim\limits_{n\to\infty}\biggl{(}\frac{[n]_{q_n}}{%
[n+1]_{q_n}+\beta_1}\biggl{)} \frac{2q_n\bigl{(}\lbrack n]_{q_n}+\beta_2%
\bigl{)}-(1+q_n)\bigl{(}\lbrack n+1]_{q_n}+\beta_1\bigl{)}}{[2]_{q_n}}x \\
&&-\alpha_2\lim\limits_{n\to\infty}\biggl{(}\frac{[n]_{q_n}}{%
[n+1]_{q_n}+\beta_1}\biggl{)}\frac{2q_n}{1+q_n} +\lim\limits_{n\to\infty}%
\biggl{(}\alpha_1+\frac{1}{[2]_{q_n}}\biggl{)}\frac{[n]_{q_n}}{%
[n+1]_{q_n}+\beta_1} \\
&=&\lim\limits_{n\to\infty}\frac{q_n^{n+1}-q_n^{n+2}-q_n+1+2q_n%
\beta_2(q_n-1)+\beta_1(1-q_n^2)}{q_n-1}\frac{x}{[2]_{q_n}}%
-\alpha_2+\alpha_1+\frac12 \\
&=&-\frac{1+a+2(\beta_1-\beta_2)}{2}x+\frac{1+2(\alpha_1-\alpha_2)}{2}.
\end{eqnarray*}
\begin{eqnarray*}
\lim\limits_{n\to\infty}[n]_{q_n}K^{(\alpha,\beta)}_{n,q_n}\bigl{(}(t-x)^2;x%
\bigl{)} &=&\lim\limits_{n\to\infty}[n]_{q_n}\bigl{(}K^{(\alpha,%
\beta)}_{n,q_n}(t^2;x)-x^2-2xK^{(\alpha,\beta)}_{n,q_n}(t-x;x)\bigl{)} \\
&=& \lim\limits_{n\to\infty}[n]_{q_n}\biggl{\{}\biggl{(}1-\frac1{[n]_{q_n}}%
\biggl{)}\biggl{(}1+\frac{(q_n-1)^2}{[3]_{q_n}}+\frac{2(q_n-1)}{[2]_{q_n}}%
\biggl{)} \biggl{(}\frac{[n]_{q_n}+\beta_2}{[n+1]_{q_n}+\beta_1}\biggl{)}^2-1%
\biggl{\}}x^2 \\
&&+\lim\limits_{n\to\infty}[n]_{q_n}\biggl{\{}-\biggl{(}1-\frac1{[n]_{q_n}}%
\biggl{)}\biggl{(}1+\frac{(q_n-1)^2}{[3]_{q_n}} +\frac{2(q_n-1)}{[2]_{q_n}}%
\biggl{)}\biggl{(}\frac{[n]_{q_n}+\beta_2}{[n+1]_{q_n}+\beta_1}\biggl{)}^2%
\frac{[2]_{q_n}\alpha_2}{[n]_{q_n}+\beta_2} \\
&&+\biggl{(}1+\frac{q_n^2-1}{[3]_{q_n}}+\frac{2q_n(2\alpha_1+1)}{1+q_n}%
\biggl{)}\frac{[n]_{q_n}+\beta_2}{\bigl{(}\lbrack n+1]_{q_n}+\beta_1\bigl{)}%
^2}\biggl{\}}x \\
&& +\lim\limits_{n\to\infty}[n]_{q_n}\biggl{(}1-\frac1{[n]_{q_n}}\biggl{)}%
\biggl{(}1+\frac{(q_n-1)^2}{[3]_{q_n}}+\frac{2(q_n-1)}{[2]_{q_n}}\biggl{)} %
\biggl{(}\frac{[n]_{q_n}+\beta_2}{[n+1]_{q_n}+\beta_1}\biggl{)}^2\frac{%
q_n\alpha_2^2}{([n]_{q_n}+\beta_2)^2} \\
&& -\lim\limits_{n\to\infty}[n]_{q_n}\biggl{(}1+\frac{q_n^2-1}{[3]_{q_n}}+%
\frac{2q_n(2\alpha_1+1)}{1+q_n}\biggl{)}\frac{[n]_{q_n}+\beta_2}{\bigl{(}%
\lbrack n+1]_{q_n} +\beta_1\bigl{)}^2}\frac{\alpha_2}{[n]_{q_n}+\beta_2} \\
&& -2x\lim\limits_{n\to\infty}[n]_{q_n}K^{(\alpha,\beta)}_{n,q_n}(t-x;x) \\
&=& -2\alpha_2x-x^2+2x+2\alpha_1x-2x\biggl{(}-\frac{1+a+2(\beta_1-\beta_2)}{2%
}x+\frac{1+2(\alpha_1-\alpha_2)}{2}\biggl{)} \\
&=& \bigl{(}a+2\beta_1-2\beta_2\bigl{)}x^2+x.
\end{eqnarray*}
\end{proof}
\end{lem}

\section{Convergence results}

First we give the following theorem on convergence of $K^{(\alpha,%
\beta)}_{n,q}(f;x)$ to $f(x)$.

\begin{thm}
\label{3.1}  Let $q=q_n\in(0,1)$ be a sequence such that $q_n\to1$ as $n\to\infty$ and $f$ a continuous function on [0,1]. Then
\begin{equation*}
\lim\limits_{n\to\infty}\max\limits_{\frac{\alpha_2}{[n]_{q_n}+\beta_2}\leq x\leq%
\frac{[n]_{q_n}+\alpha_2}{[n]_{q_n}+\beta_2}}\mid
K^{(\alpha,\beta)}_{n,q_n}(f;x)-f(x)\mid=0
\end{equation*}
\end{thm}

\begin{proof}
Taking into consideration the equalities in Lemma 1, for $v=0,1,2$ we can
write \newline
\begin{equation}  \label{3.1}
\lim\limits_{n\to\infty}\max\limits_{\frac{\alpha_2}{[n]_{q_n}+\beta_2}\leq x\leq%
\frac{[n]_{q_n}+\alpha_2}{[n]_{q_n}+\beta_2}}\left|
K^{(\alpha,\beta)}_{n,q_n}(t^v;x)-x^v\right|=0.
\end{equation}
Now consider the sequence of operators \newline
\newline
$~~~~~~~~~~~~~~~~~~~~~~K^*_{n,q_n}(f;x)=\left\{
\begin{array}{ll}
K^{(\alpha,\beta)}_{n,q_n}, & \hbox{if $\frac{\alpha_2}{[n]_{q_n}+\beta_2}\leq
x\leq\frac{[n]_{q_n}+\alpha_2}{[n]_{q_n}+\beta_2}$ ;} \\
f(x), & \hbox{if
$x\in\left[0,\frac{\alpha_2}{[n]_{q_n}+\beta_2}\right]\cup\left[\frac{[n]_{q_n}+%
\alpha_2}{[n]_{q_n}+\beta_2},1\right]$.}%
\end{array}
\right.$\newline
\newline
Then obviously,\newline
\begin{equation}  \label{3.2}
\|K^*_{n,q_n}f-f\|=\max\limits_{\frac{\alpha_2}{[n]_{q_n}+\beta_2}\leq x\leq\frac{%
[n]_{q_n}+\alpha_2}{[n]_{q_n}+\beta_2}}\mid K^{(\alpha,\beta)}_{n,q_n}(f;x)-f(x)\mid
\end{equation}
and using (3.1) we obtain\newline
\begin{equation*}
\lim\limits_{n\to\infty}\|K^*_{n,q_n}(t^v;x)-x^v\|_{C[0,1]}=0,~~v=0,1,2
\end{equation*}
Applying the Korovkin theorem \cite{kor} (see also \cite{cam}) to the
sequence of positive linear operators $K^*_{n,q_n}$, we obtain
\begin{equation*}
\lim\limits_{n\to\infty}\|K^*_{n,q_n}(f;x)-f(x)\|_{C[0,1]}=0
\end{equation*}
for every continuous function $f$. Therefore (3.2) gives
\begin{equation*}
\lim\limits_{n\to\infty}\max\limits_{\frac{\alpha_2}{[n]_{q_n}+\beta_2}\leq x\leq%
\frac{[n]_{q_n}+\alpha_2}{[n]_{q_n}+\beta_2}}\mid
K^{(\alpha,\beta)}_{n,q_n}(f;x)-f(x)\mid=0
\end{equation*}
and thus the result is obtained.
\end{proof}

We use modulus of continuity to give quantitative error estimates for the
approximation by positive linear operators.

\begin{thm}
\label{3.2}  If $f\in C[0,1]$ and $0<q<1$, then\newline
\begin{eqnarray*}
\|K^{(\alpha,\beta)}_{n,q}(f;x)-f(x)\|&\leq&2\omega_f(\delta_n),
\end{eqnarray*}
where
\begin{eqnarray*}
\delta_n^2&=&\frac{2q^2(2q+1)}{[2][3]}\frac{[n]([n]+\alpha_2)}{([n+1]+\beta_1)^2}
   +\frac{q}{1+q}\left(\frac{3+5q+4q^2}{1+q+q^2}+4\alpha_1\right)\frac{[n]}{([n+1]+\beta_1)^2}\\
   &&-\frac2{1+q}\frac{(2q[n]+2\alpha_1+1)([n]+\alpha_2)}{([n+1]+\beta_1)([n]+\beta_2)} +\left(\frac{[n]+\alpha_2}{[n]+\beta_2}\right)^2+\left(\frac{1+\alpha_1}{[n+1]+\beta_1}\right)^2.
\end{eqnarray*}
\end{thm}

\begin{proof}
For any $x,y\in[a,b]$, it is known that
\begin{eqnarray*}
|f(y)-f(x)|&\leq &\omega_f(\delta)\biggl{(}\frac{(y-x)^2}{\delta^2}+1%
\biggl{)}.
\end{eqnarray*}
Therefore, we get
\begin{eqnarray*}
|K^{(\alpha,\beta)}_{n,q}(f;x)-f(x)|\leq
K^{(\alpha,\beta)}_{n,q}(|f(t)-f(x)|;x)&\leq & \omega_f(\delta)\biggl{(}%
1+\frac1{\delta^2}K^{(\alpha,\beta)}_{n,q}\bigl{(}(t-x)^2;x\bigl{)}\biggl{)}
\end{eqnarray*}

By using Lemma 2.2, we can write
\begin{eqnarray*}
|K^{(\alpha,\beta)}_{n,q}(f;x)-f(x)| &\leq& \omega_f(\delta)\biggl{[}
1+\frac1{\delta^2}\biggl{\{}\frac{2q^2(2q+1)}{[2][3]}\frac{[n]([n]+\alpha_2)}{([n+1]+\beta_1)^2}
   +\frac{q}{1+q}\left(\frac{3+5q+4q^2}{1+q+q^2}+4\alpha_1\right)\frac{[n]}{([n+1]+\beta_1)^2} \\
   &&-\frac2{1+q}\frac{(2q[n]+2\alpha_1+1)([n]+\alpha_2)}{([n+1]+\beta_1)([n]+\beta_2)} +\left(\frac{[n]+\alpha_2}{[n]+\beta_2}\right)^2+\left(\frac{1+\alpha_1}{[n+1]+\beta_1}\right)^2\biggl{\}}\biggl{]}
\end{eqnarray*}
Choosing
\begin{eqnarray*}
\delta=\delta_n&=&\biggl{\{}\frac{2q^2(2q+1)}{[2][3]}\frac{[n]([n]+\alpha_2)}{([n+1]+\beta_1)^2}
   +\frac{q}{1+q}\left(\frac{3+5q+4q^2}{1+q+q^2}+4\alpha_1\right)\frac{[n]}{([n+1]+\beta_1)^2}\\
   &&-\frac2{1+q}\frac{(2q[n]+2\alpha_1+1)([n]+\alpha_2)}{([n+1]+\beta_1)([n]+\beta_2)} +\left(\frac{[n]+\alpha_2}{[n]+\beta_2}\right)^2+\left(\frac{1+\alpha_1}{[n+1]+\beta_1}\right)^2\biggl{\}}^{\frac12},
\end{eqnarray*}
 we have

\begin{eqnarray*}
\|K^{(\alpha,\beta)}_{n,q}(f;x)-f(x)\|&\leq&2\omega_f(\delta_n),
\end{eqnarray*}

Thus, we obtain the desired result.
\end{proof}

\section{Local Approximation}

We begin considering the following $K$-functional:
\begin{equation*}
K_2(f,\delta^2):=\inf\bigl{\{}\|f-g\|+\delta^2\|g^{\prime \prime }\|,~~g\in C^2[0,1]%
\bigl{\}}, ~~~~\delta\geq0,
\end{equation*}
where
\begin{equation*}
C^2[0,1]:=\bigl{\{}g:~g,g^{\prime },g^{\prime \prime }\in C[0,1]\bigl{\}}.
\end{equation*}
Then, in view of a known result \cite{tot}, there exists an absolute
constant $C_0>0$ such that
\begin{equation*}
K_2(f,\delta^2)\leq C_0\omega_2(f,\delta)
\end{equation*}
where
\begin{equation*}
\omega_2(f,\delta):=\sup\limits_{0<h\leq\delta}\sup\limits_{x\pm h\in[0,1]}%
\bigl{|}f(x-h)-2f(x)+f(x+h)\bigl{|}
\end{equation*}
is the second modulus of smoothness of $f\in C[0,1]$.

\begin{thm}
\label{3.3} Let $f\in C[0,1]$ with $0<q<1$. Then for every $x\in\left[%
\frac{\alpha_2}{[n]+\beta_2},\frac{\left[n\right]+\alpha_2}{\left[n\right]+\beta_2}\right]$, we
have
\begin{equation*}
\bigl{|}K_{n,q}^{(\alpha,\beta)}(f;x)-f(x)\bigl{|}\leq C\omega_2\bigl{(}f,%
\sqrt{\delta_n(x)}\bigl{)}+\omega\bigl{(}f,|(a_n-1)x+b_n|\bigl{)}
\end{equation*}
where $a_n=\frac{2q}{1+q}\frac{[n]+\beta_2}{[n+1]+\beta_1}$, $%
b_n=\frac1{[n+1]+\beta_1}\left(\alpha_1+\frac1{1+q}\right)-\frac{2q}{1+q%
}\frac{\alpha_2}{[n+1]+\beta_1}$ and\newline
\begin{eqnarray*}
\delta_n(x)&=&\left\{\frac{1+2q+4q^2+5q^3}{1+2q+2q^2+q^3}\left(\frac{[n]+\beta_2}{[n+1]+\beta_1}\right)^2
-\frac{2(3q+1)}{1+q}\frac{[n]+\beta_2}{[n+1]+\beta_1}+2\right\}x^2\\
&&+\left\{\left(\frac{5+7q+6q^2}{1+q+q^2}+\frac{2q^2(2q+1)\alpha_2}{(1+q+q^2)[n]}+4\alpha_1\right)\frac{[n]+\beta_2}{([n+1]+\beta_1)^2}
+2\frac{\alpha_2}{[n+1]+\beta_1}\right\}x\\
&&+\frac{q^2(2q+1)}{1+q+q^2}\left(\frac{\alpha_2}{[n+1]+\beta_1}\right)^2
-\frac{q}{1+q}\left(\frac{3+5q+4q^2}{1+q+q^2}+4\alpha_1\right)\frac{\alpha_2}{([n+1]+\beta_1)^2}+2\left(\frac{1+\alpha_1}{[n+1]+\beta_1}\right)^2.
\end{eqnarray*}
\end{thm}

\begin{proof}
Let
\begin{equation*}
\bar{K}_{n,q}^{(\alpha,\beta)}(f;x)=K_{n,q}^{(\alpha,%
\beta)}(f;x)+f(x)-f(a_nx+b_n)
\end{equation*}
where $f\in C[0,1]$, $a_n=\frac{2q}{1+q}\frac{[n]+\beta_2}{[n+1]+\beta_1}$
and $b_n=\frac{1}{[n+1]+\beta_1}\left(\alpha_1+\frac{1}{1+q}\right)-%
\frac{2q}{1+q}\frac{\alpha_2}{[n+1]+\beta_1}$. Using the Taylor formula
\begin{equation*}
g(t)=g(x)+g^{\prime t}_{x}(t-s)g^{\prime \prime 2}[0,1],
\end{equation*}
we have
\begin{equation*}
\bar{K}_{n,q}^{(\alpha,\beta)}(g;x)=g(x)+{K}_{n,q}^{(\alpha,\beta)}\biggl{(}%
\int^{t}_{x}(t-s)g^{\prime \prime }(s)ds;x\biggl{)} -%
\int^{a_nx+b_n}_{x}(a_nx+b_n-s)g^{\prime \prime 2}[0,1].
\end{equation*}
Hence
\begin{eqnarray*}
\bigl{|}\bar{K}_{n,q}^{(\alpha,\beta)}(g;x)-g(x)\bigl{|} &\leq& {K}%
_{n,q}^{(\alpha,\beta)}\biggl{(}\biggl{|}\int^{t}_{x}(t-s)g^{\prime \prime
}(s)ds\biggl{|};x\biggl{)} + \biggl{|}\int^{a_nx+b_n}_{x}|a_nx+b_n-s)|~%
\bigl{|}g^{\prime \prime }(s)\bigl{|}ds\biggl{|}  \notag \\
&\leq & {K}_{n,q}^{(\alpha,\beta)}\bigl{(}(t-x)^2;x\bigl{)}\|g^{\prime
\prime }\|+(a_nx+b_n-x)^2\|g^{\prime \prime }\|\\
&=& \biggl{[}\biggl{\{}\frac{q[n-1]
}{[n]}\biggl{(}1+\frac{(q-1)^2}{[3]}+\frac{2(q-1)}{[2]}\biggl{)}\biggl{(}%
\frac{[n]+\beta_2}{[n+1]+\beta_1}\biggl{)}^2 -\frac{4q}{1+q}\frac{[n]+\beta_2%
}{[n+1]+\beta_1}+1\biggl{\}}x^2 \\
&& +\biggl{\{}\biggl{(}1+\frac{q^2-1}{[3]}+(2\alpha_1+1)\frac{2q}{1+q}%
\biggl{)}\frac{[n]+\beta_2}{([n+1]+\beta_1)^2} \\
&& -\frac{q[n-1]}{[n]}\frac{[2]\alpha_2}{[n]+\beta_2}\biggl{(}1+\frac{(q-1)^2%
}{[3]} +\frac{2(q-1)}{[2]}\biggl{)}\biggl{(}\frac{[n]+\beta_2}{[n+1]+\beta_1}%
\biggl{)}^2 \\
&& +\frac{2q}{1+q}\frac{2\alpha_2}{[n]+\beta_2}\biggl{(}\frac{[n]+\beta_2}{%
[n+1]+\beta_1}\biggl{)} -\frac2{[n+1]+\beta_1}\biggl{(}\alpha_1+\frac1{[2]}%
\biggl{)}\biggl{\}}x \\
&&+\frac{q^2[n-1]}{[n]}\biggl{(}1+\frac{(q-1)^2}{[3]}+\frac{2(q-1)}{[2]}%
\biggl{)}\biggl{(}\frac{\alpha_2}{[n+1]+\beta_1}\biggl{)}^2 \\
&&-\biggl{(}1+\frac{q^2-1}{[3]}+(2\alpha_1+1)\frac{2q}{1+q}\biggl{)}\frac{%
\alpha_2}{([n+1]+\beta_1)^2} +\frac1{([n+1]+\beta_1)^2}\biggl{(}\alpha_1^2+%
\frac{2\alpha_1}{[2]}+\frac1{[3]}\biggl{)}\\
&&+\left\{\frac{2q}{1+q}\frac{[n]+\beta_2}{[n+1]+\beta_1}x
+\frac{1}{[n+1]+\beta_1}\left(\alpha_1+\frac{1}{1+q}\right)-
\frac{2q}{1+q}\frac{\alpha_2}{[n+1]+\beta_1}-x\right\}^2\biggl{]}\|g^{\prime\prime}\|\\
&=&\left\{\left(1-\frac1{[n]}\right)\frac{2q^2(2q+1)}{[2][3]}
\left(\frac{[n]+\beta_2}{[n+1]+\beta_1}\right)^2-\frac{4q}{1+q}\frac{[n]+\beta_2}{[n+1]+\beta_1}
+1+\left(\frac{[n]+\beta_2}{[n+1]+\beta_1}-1\right)^2\right\}x^2\\
&&+\biggl{\{}\frac{q}{1+q}\left(\frac{3+5q+4q^2}{1+q+q^2}+4\alpha_1\right)\frac{[n]+\beta_2}{([n+1]+\beta_1)^2}
-\left(1-\frac1{[n]}\right)\frac{\alpha_2}{[n]+\beta_2}\frac{2q^2(2q+1)}{[3]}\left(\frac{[n]+\beta_2}{[n+1]+\beta_1}\right)^2\\
&&+\frac{4q}{1+q}\frac{\alpha_2}{[n+1]+\beta_1}-\frac{2}{[n+1]+\beta_1}\left(\alpha_1+\frac1{[2]}\right)
+2\left(\frac{[n]+\beta_2}{[n+1]+\beta_1}-1\right)\left(\frac{1+\alpha_1}{[n+1]+\beta_1}\right)\biggl{\}}x\\
&&q\left(1-\frac1{[n]}\right)\frac{2q^2(2q+1)}{[2][3]}\left(\frac{\alpha_2}{[n+1]+\beta_1}\right)^2
-\frac{q}{1+q}\left(\frac{3+5q+4q^2}{1+q+q^2}+4\alpha_1\right)\frac{\alpha_2}{([n+1]+\beta_1)^2}\\
&&+\frac1{([n+1]+\beta_1)^2}\biggl{(}\alpha_1^2+%
\frac{2\alpha_1}{[2]}+\frac1{[3]}\biggl{)}+\left(\frac{1+\alpha_1}{[n+1]+\beta_1}\right)^2\\
&\leq&\left\{\frac{1+2q+4q^2+5q^3}{1+2q+2q^2+q^3}\left(\frac{[n]+\beta_2}{[n+1]+\beta_1}\right)^2
-\frac{2(3q+1)}{1+q}\frac{[n]+\beta_2}{[n+1]+\beta_1}+2\right\}x^2\\
&&+\left\{\left(\frac{5+7q+6q^2}{1+q+q^2}+\frac{2q^2(2q+1)\alpha_2}{(1+q+q^2)[n]}+4\alpha_1\right)\frac{[n]+\beta_2}{([n+1]+\beta_1)^2}
+2\frac{\alpha_2}{[n+1]+\beta_1}\right\}x\\
&&+\frac{q^2(2q+1)}{1+q+q^2}\left(\frac{\alpha_2}{[n+1]+\beta_1}\right)^2
-\frac{q}{1+q}\left(\frac{3+5q+4q^2}{1+q+q^2}+4\alpha_1\right)\frac{\alpha_2}{([n+1]+\beta_1)^2}+2\left(\frac{1+\alpha_1}{[n+1]+\beta_1}\right)^2\\
& =&\delta_n(x)\|g^{\prime \prime }\|(3.3)
\end{eqnarray*}
Using (3.3) and the uniform boundedness of $\bar{K}_{n,q}^{(\alpha,\beta)}$,
we get
\begin{align*}
\bigl{|}{K}_{n,q}^{(\alpha,\beta)}(f;x)-f(x)\bigl{|}&\leq\bigl{|}\bar{K}%
_{n,q}^{(\alpha,\beta)}(f-g;x)\bigl{|} +\bigl{|}\bar{K}_{n,q}^{(\alpha,%
\beta)}(g;x)-g(x)\bigl{|}+\bigl{|}f(x)-g(x)\bigl{|}+\bigl{|}f(a_nx+b_n)-f(x)%
\bigl{|} \\
&\leq4\bigl{\|}f-g\bigl{\|}+\delta_n(x)\bigl{\|}g^{\prime \prime }\bigl{\|}%
+\omega\bigl{(}f,|(a_n-1)x+b_n|\bigl{)}
\end{align*}
Taking the infimum on the right hand side over all $g\in C^2[0,1]$, we
obtain
\begin{equation*}
\bigl{|}{K}_{n,q}^{(\alpha,\beta)}(f;x)-f(x)\bigl{|}\leq C\omega_2\bigl{(}f,%
\sqrt{\delta_n(x)}\bigl{)}+\omega\bigl{(}f,|(a_n-1)x+b_n|\bigl{)}.
\end{equation*}
This completes the proof.
\end{proof}

\begin{cor}
\label{3.4} Assume that $q_n\in(0,1)$, $q_n\to1$ as $n\to\infty$. For any $%
f\in C^2[0,1]$ we have
\begin{equation*}
\lim\limits_{n\to\infty}\|K_{n,q_n}^{(\alpha,\beta)}(f)-f\|=0.
\end{equation*}
\end{cor}

Furthermore, we estimate the rate of convergence for smooth functions. For
this reason, we now state following general estimate theorem obtained by
Shisha and Mond \cite{shi} in terms of modulus of continuity.

\begin{thm}
\label{3.5} Let $[c,d]\subseteq[a,b]$ and $(L_n)_{n\in\mathbb{N}}$ be a
sequence of positive linear operators such that
\begin{equation*}
L_n:C[a,b]\to C[c,d]
\end{equation*}
If $f^{\prime }C[a,b]$ and $x\in[c,d]$, then we have
\begin{align*}
|L_n(f;x)-f(x)|&\leq|f(x)||L_n(1;x)-1|+|f^{\prime }(x)||L_n(t-x;x)|+\sqrt{L_n%
\bigl{(}(t-x)^2;x\bigl{)}} \\
&~~~\times\biggl{\{}\sqrt{L_n(1;x)}+\frac1\delta\sqrt{L_n\bigl{(}(t-x)^2;x%
\bigl{)}}\biggl{\}}\omega(f^{\prime }\delta).
\end{align*}
where $\omega$ is the modulus of continuity of the function $f$ defined by
\begin{equation*}
\omega(f;\delta)=\sup\bigl{\{}|f(x)-f(y)|:~x,y\in[0,1],~|x-y|\leq\delta%
\bigl{\}}
\end{equation*}
for any positive number $\delta$.
\end{thm}

\begin{thm}
\label{3.6} For any $f\in C^1[0,1]$ and each $x\in\left[\frac{\alpha_2}{%
[n]+\beta_2},\frac{[n]+\alpha_2}{[n]+\beta_2}\right]$, we have
\begin{equation*}
\bigl{|}K_{n,q}^{(\alpha,\beta)}-f(x)\bigl{|}\leq\biggl{|}\biggl{(}\frac{2q}{%
1+q}\frac{[n]+\beta_2}{[n+1]+\beta_1}-1\biggl{)}x +\frac{1+\alpha_1+q%
\alpha_1-2q\alpha_2}{(1+q)([n+1]+\beta_1)}\biggl{|}|f^{\prime }(x)|+2\sqrt{%
\delta_n(x)}\omega(f^{\prime },\sqrt{\delta_n(x)})
\end{equation*}
\end{thm}

\begin{proof}
In view of Lemma 2.1, Lemma 2.2 $\&$ Theorem 4.3, and if we choose $\delta=\sqrt{%
\delta_n(x)}=\sqrt{K_{n,q}^{(\alpha,\beta)}\bigl{(}(t-x)^2;x\bigl{)}}$, we
have
\begin{align*}
\bigl{|}K_{n,q}^{(\alpha,\beta)}-f(x)\bigl{|}&\leq\biggl{|}\biggl{(}\frac{2q%
}{1+q}\frac{[n]+\beta_2}{[n+1]+\beta_1}-1\biggl{)}x +\frac{%
1+\alpha_1+q\alpha_1-2q\alpha_2}{(1+q)([n+1]+\beta_1)}\biggl{|}|f^{\prime
}(x)| \\
&+\sqrt{K_{n,q}^{(\alpha,\beta)}\bigl{(}(t-x)^2;x\bigl{)}} \times\biggl{\{}%
1+\frac1{\delta}\sqrt{K_{n,q}^{(\alpha,\beta)}\bigl{(}(t-x)^2;x\bigl{)}}%
\biggl{\}}\omega(f^{\prime },\delta) \\
&=\biggl{|}\biggl{(}\frac{2q}{1+q}\frac{[n]+\beta_2}{[n+1]+\beta_1}-1%
\biggl{)}x +\frac{1+\alpha_1+q\alpha_1-2q\alpha_2}{(1+q)([n+1]+\beta_1)}%
\biggl{|}|f^{\prime }(x)|+2\sqrt{\delta_n(x)}\omega(f^{\prime },\sqrt{%
\delta_n(x)}).
\end{align*}
\end{proof}

\section{Voronovskaja type Theorem}

Next we prove Voronovskaja type result for Kantorovich type $q$%
-Bernstein-Stancu operators.\newline

\begin{thm}
\label{3.7} Assume that $q=q_n\in(0,1)$, $q_n\to1$ and $q_n^n\to a~(0\leq a<1)$ as $n\to
\infty$. For any $f\in C^2[0,1]$ the following equality holds\newline
\begin{equation*}
\lim\limits_{n\to\infty}[n]_{q_n}\left(K^{(\alpha,%
\beta)}_{n,q_n}(f;x)-f(x)\right) =f^{\prime }(x)\biggl{(}-\frac{%
1+a+2(\beta_1-\beta_2)}{2}x+\frac{1+2(\alpha_1-\alpha_2)}{2}\biggl{)}%
+\frac12f^{\prime \prime }(x)\left((a+2\beta_1-2\beta_2)x^2+x\right)
\end{equation*}
uniformly on $x\in\left[\frac{\alpha_2}{[n]_{q_n}+\beta_2},\frac{[n]_{q_n}+\alpha_2}{%
[n]_{q_n}+\beta_2}\right].$
\end{thm}

\begin{proof}
Let $f\in C^2[0,1]$ and $x\in[0,1]$ be fixed. By the Taylor formula we may
write
\begin{equation}  \label{2.5}
f(t)=f(x)+f^{\prime }(x)(t-x)+\frac12f^{\prime \prime 2}+r(t;x)(t-x)^2,
\end{equation}
where $r(t;x)$ is the Peano form of remainder, $r(.;x)\in C[0,1]$ and $%
\lim\limits_{t\to x}=0$. Applying $K^{(\alpha,\beta)}_{n,q_n}(f;x)$ on both
sides of (3.4), we obtain
\begin{eqnarray*}
[n]_{q_n}\bigl{(}K^{(\alpha,\beta)}_{n,q_n}(f;x)-f(x)\bigl{)}&=&f^{\prime
}(x)[n]_{q_n}K^{(\alpha,\beta)}_{n,q_n}\bigl{(}(t-x);x\bigl{)}
+\frac12f^{\prime \prime }(x)[n]_{q_n}K^{(\alpha,\beta)}_{n,q_n}\bigl{(}%
(t-x)^2;x\bigl{)} \\
&& +[n]_{q_n}K^{(\alpha,\beta)}_{n,q_n}\bigl{(}r(t;x)(t-x)^2;x\bigl{)}.
\end{eqnarray*}
By the Cauchy-Schwartz inequality, we have
\begin{equation}  \label{2.6}
K^{(\alpha,\beta)}_{n,q_n}\bigl{(}r(t;x)(t-x)^2;x\bigl{)} \leq\sqrt{%
K^{(\alpha,\beta)}_{n,q_n}\bigl{(}r^2(t;x);x\bigl{)}}\sqrt{%
K^{(\alpha,\beta)}_{n,q_n}\bigl{(}(t-x)^4;x\bigl{)}}
\end{equation}
Observe that $r^2(x,x)=0$ and $r^2(.;x)\in C[0,1]$. Then it follows from
Corollary 3.4 that
\begin{equation}  \label{3.5}
\lim\limits_{n\to\infty}K^{(\alpha,\beta)}_{n,q_n}\bigl{(}r^2(t;x);x\bigl{)}%
=r^2(x,x)=0
\end{equation}
uniformly with respect to $x\in\left[\frac{\alpha_2}{[n]_{q_n}+\beta_2},\frac{%
[n]_{q_n}+\alpha_2}{[n]_{q_n}+\beta_2}\right]$. Now from (3.5), (3.6) and Lemma 3.4 we
get immediately
\begin{equation*}
\lim\limits_{n\to\infty}[n]_{q_n}K^{(\alpha,\beta)}_{n,q_n}\left(%
r(t;x)(t-x)^2;x\right)=0.
\end{equation*}
The proof is completed.
\end{proof}

$~~~~~~~~~~$\parindent=8mmNow we give the rate of convergence of the
operators $K^{(\alpha,\beta)}_{n,q}$ in terms of the elements of the usual
Lipschitz class $Lip_M(\alpha)$.\newline
$~~~~~~~~~~$Let $f\in C[0,1]$, $M>0$ and $0<\alpha\leq1$.We recall that $f$
belongs to the class $Lip_M(\alpha)$ if the inequality
\begin{equation*}
|f(t)-f(x)|\leq M|t-x|^\alpha~~~(t,x\in[0,1])
\end{equation*}
is satisfied.\newline

\begin{thm}
\label{3.8} Let $q=q_n\in(0,1)$ such that $\lim\limits_{n\to\infty}q_n=1$.
Then for each $f\in Lip_M(\alpha)$ we have
\begin{equation*}
\|K^{(\alpha,\beta)}_{n,q_n}-f(x)\|\leq M\delta_n^{\alpha}
\end{equation*}
where $\|.\|$ is the supremum norm over $\left[\frac{\alpha_2}{%
[n]_{q_n}+\beta_2},\frac{[n]_{q_n}+\alpha_2}{[n]_{q_n}+\beta_2}\right]$ and
\begin{eqnarray*}
\bigl{\|}K^{(\alpha,\beta)}_{n,q_n}(f)-f\bigl{\|}&\leq& M \biggl{[}\frac{2q_n^2(2q_n+1)}{[2]_{q_n}[3]_{q_n}}\frac{[n]_{q_n}([n]_{q_n}+\alpha_2)}{([n+1]_{q_n}+\beta_1)^2}
   +\frac{q_n}{1+q_n}\left(\frac{3+5q_n+4q_n^2}{1+q_n+q_n^2}+4\alpha_1\right)\frac{[n]_{q_n}}{([n+1]_{q_n}+\beta_1)^2}\\
  &&-\frac{2}{1+q_n}\frac{(2q_n[n]_{q_n}+2\alpha_1+1)([n]_{q_n}+\alpha_2)}{([n+1]_{q_n}+\beta_1)([n]_{q_n}+\beta_2)}
  +\left(\frac{[n]_{q_n}+\alpha_2}{[n]_{q_n}+\beta_2}\right)^2+\left(\frac{1+\alpha_1}{[n+1]_{q_n}+\beta_1}\right)^2\biggl{]}^{\frac{\alpha}{2}}.
\end{eqnarray*}
\end{thm}

\begin{proof}
. Let us denote $P_{n,k}^{(\alpha ,\beta )}(x)=\biggl{(}\frac{%
[n]_{q_{n}}+\beta _{2}}{[n]_{q_{n}}}\biggl{)}^{n}~\sum\limits_{k=0}^{n}\left[
\begin{array}{c}
n \\
k%
\end{array}%
\right] _{q_{n}}\biggl{(}x-\frac{\alpha _{2}}{[n]_{q_{n}}+\beta _{2}}%
\biggl{)}_{q_{n}}^{k}~\biggl{(}\frac{[n]_{q_{n}}+\alpha _{2}}{%
[n]_{q_{n}}+\beta _{2}}-x\biggl{)}_{q_{n}}^{n-k}$. Then by the monotonicity
of the operators $K_{n,q_n}^{(\alpha ,\beta )}$, we can write
\begin{eqnarray*}
\bigl{|}K_{n,q_{n}}^{(\alpha ,\beta )}-f(x)\bigl{|} &\leq
&K_{n,q_{n}}^{(\alpha ,\beta )}\bigl{(}|f(t)-f(x)|;x\bigl{)} \\
&\leq &\sum\limits_{k=0}^{n}P_{n,k}^{(\alpha ,\beta )}(x)\int_{0}^{1}%
\biggl{|}f\biggl{(}\frac{[k]_{q_{n}}+q_n^{k}t+\alpha _{1}}{[n+1]_{q_{n}}+\beta
_{1}}\biggl{)-f(x)\biggl{|}}d_{q_n}t \\
&\leq &M\sum\limits_{k=0}^{n}P_{n,k}^{(\alpha ,\beta )}(x)\int_{0}^{1}%
\biggl{|}\frac{[k]_{q_{n}}+q_n^{k}t+\alpha _{1}}{[n+1]_{q_{n}}+\beta _{1}}-x%
\biggl{|}^{\alpha }d_{q_n}t.
\end{eqnarray*}

On the other hand, by using the Hölder's inequality for integrals with $p=%
\frac{2}{\alpha}$ and $q=\frac{2}{2-\alpha}$, we have

\begin{eqnarray*}
\bigl{|}K^{(\alpha,\beta)}_{n,q_n}-f(x)\bigl{|} &\leq&
M\sum\limits_{k=0}^{n}P^{(\alpha,\beta)}_{n,k}(x)\biggl{\{}\int^{1}_{0} %
\biggl{(}\frac{[k]_{q_n}+q_n^kt+\alpha_1}{[n+1]_{q_n}+\beta_1}-x\biggl{)}^2d_{q_n}t%
\biggl{\}}^{\frac{\alpha}{2}}\biggl{\{}\int^{1}_{0}1~d_{q_n}t\biggl{\}}^\frac{%
2-\alpha}{2} \\
&=& M\sum\limits_{k=0}^{n}\biggl{\{}P^{(\alpha,\beta)}_{n,k}(x)~\int^{1}_{0} %
\biggl{(}\frac{[k]_{q_n}+q_n^kt+\alpha_1}{[n+1]_{q_n}+\beta_1}-x\biggl{)}^2d_{q_n}t%
\biggl{\}}^{\frac{\alpha}{2}}P^{(\alpha,\beta)}_{n,k}(x)^\frac{2-\alpha}{2}.
\end{eqnarray*}
Now again applying the Hölder's inequality for the sum with $p=\frac2{\alpha}
$ and $q=\frac2{2-\alpha}$ and taking into consideration Lemma 2.1(i) and
Lemma 2.2, we have
\begin{eqnarray*}
\bigl{|}K^{(\alpha,\beta)}_{n,q_n}(f;x)-f(x)\bigl{|} &\leq& M\biggl{(}%
K^{(\alpha,\beta)}_{n,q_n}\bigl{(}(t-x)^2;x\bigl{)}\biggl{)}^{\frac{\alpha}{2%
}}\biggl{(}K^{(\alpha,\beta)}_{n,q_n}\bigl{(}1;x\bigl{)}\biggl{)}^{\frac{%
2-\alpha}{2}} \\
&\leq& M\biggl{[}\biggl{\{}\frac{q_n[n-1]_{q_n}%
}{[n]_{q_n}}\biggl{(}1+\frac{(q_n-1)^2}{[3]_{q_n}}+\frac{2(q_n-1)}{[2]_{q_n}}\biggl{)}\biggl{(}%
\frac{[n]_{q_n}+\beta_2}{[n+1]_{q_n}+\beta_1}\biggl{)}^2 -\frac{4q_n}{1+q_n}\frac{[n]_{q_n}+\beta_2%
}{[n+1]_{q_n}+\beta_1}+1\biggl{\}}x^2 \\
&& +\biggl{\{}\biggl{(}1+\frac{q_n^2-1}{[3]_{q_n}}+(2\alpha_1+1)\frac{2q_n}{1+q_n}%
\biggl{)}\frac{[n]_{q_n}+\beta_2}{([n+1]_{q_n}+\beta_1)^2} \\
&& -\frac{q_n[n-1]_{q_n}}{[n]_{q_n}}\frac{[2]_{q_n}\alpha_2}{[n]_{q_n}+\beta_2}\biggl{(}1+\frac{(q_n-1)^2%
}{[3]_{q_n}} +\frac{2(q_n-1)}{[2]_{q_n}}\biggl{)}\biggl{(}\frac{[n]_{q_n}+\beta_2}{[n+1]_{q_n}+\beta_1}%
\biggl{)}^2 \\
&& +\frac{2q_n}{1+q_n}\frac{2\alpha_2}{[n]_{q_n}+\beta_2}\biggl{(}\frac{[n]_{q_n}+\beta_2}{%
[n+1]_{q_n}+\beta_1}\biggl{)} -\frac2{[n+1]_{q_n}+\beta_1}\biggl{(}\alpha_1+\frac1{[2]_{q_n}}%
\biggl{)}\biggl{\}}x \\
&&+\frac{q_n^2[n-1]_{q_n}}{[n]_{q_n}}\biggl{(}1+\frac{(q_n-1)^2}{[3]_{q_n}}+\frac{2(q_n-1)}{[2]_{q_n}}%
\biggl{)}\biggl{(}\frac{\alpha_2}{[n+1]_{q_n}+\beta_1}\biggl{)}^2 \\
&&-\biggl{(}1+\frac{q_n^2-1}{[3]_{q_n}}+(2\alpha_1+1)\frac{2q_n}{1+q_n}\biggl{)}\frac{%
\alpha_2}{([n+1]_{q_n}+\beta_1)^2} +\frac1{([n+1]_{q_n}+\beta_1)^2}\biggl{(}\alpha_1^2+%
\frac{2\alpha_1}{[2]_{q_n}}+\frac1{[3]_{q_n}}\biggl{)}\biggl{]}^{%
\frac{\alpha}{2}}.
\end{eqnarray*}

Replacing $x$ by $\frac{[n]_{q_n}+\alpha_2}{[n]_{q_n}+\beta_2}$ implies that
\begin{eqnarray*}
\bigl{\|}K^{(\alpha,\beta)}_{n,q_n}(f)-f\bigl{\|}&\leq& M \biggl{[}\frac{2q_n^2(2q_n+1)}{[2]_{q_n}[3]_{q_n}}\frac{[n]_{q_n}([n]_{q_n}+\alpha_2)}{([n+1]_{q_n}+\beta_1)^2}
   +\frac{q_n}{1+q_n}\left(\frac{3+5q_n+4q_n^2}{1+q_n+q_n^2}+4\alpha_1\right)\frac{[n]_{q_n}}{([n+1]_{q_n}+\beta_1)^2}\\
  &&-\frac{2}{1+q_n}\frac{(2q_n[n]_{q_n}+2\alpha_1+1)([n]_{q_n}+\alpha_2)}{([n+1]_{q_n}+\beta_1)([n]_{q_n}+\beta_2)}
  +\left(\frac{[n]_{q_n}+\alpha_2}{[n]_{q_n}+\beta_2}\right)^2+\left(\frac{1+\alpha_1}{[n+1]_{q_n}+\beta_1}\right)^2\biggl{]}^{\frac{\alpha}{2}}.
\end{eqnarray*}
Hence if we choose $\delta := \delta_n$, then we arrive at the desired
result.\newline
\end{proof}
\vspace{1cm}
\section{Graphical analysis}

With the help of Matlab, we show comparisons and some illustrative graphics
\cite{ma1} for the convergence of operators $(2.1)$ to the function $%
f(x)=1-cos(4 e^{x} )$ under different parameters.\newline

From figure $4.1,~4.2,~4.3,$ we can observe that as the value the $n$
increases, Kantorovich type $q$-Bernstein-Stancu operators given by $(2.1)$
converges towards the function.
\begin{figure*}[htb!]
\begin{center}
\includegraphics[height=6cm, width=10cm]{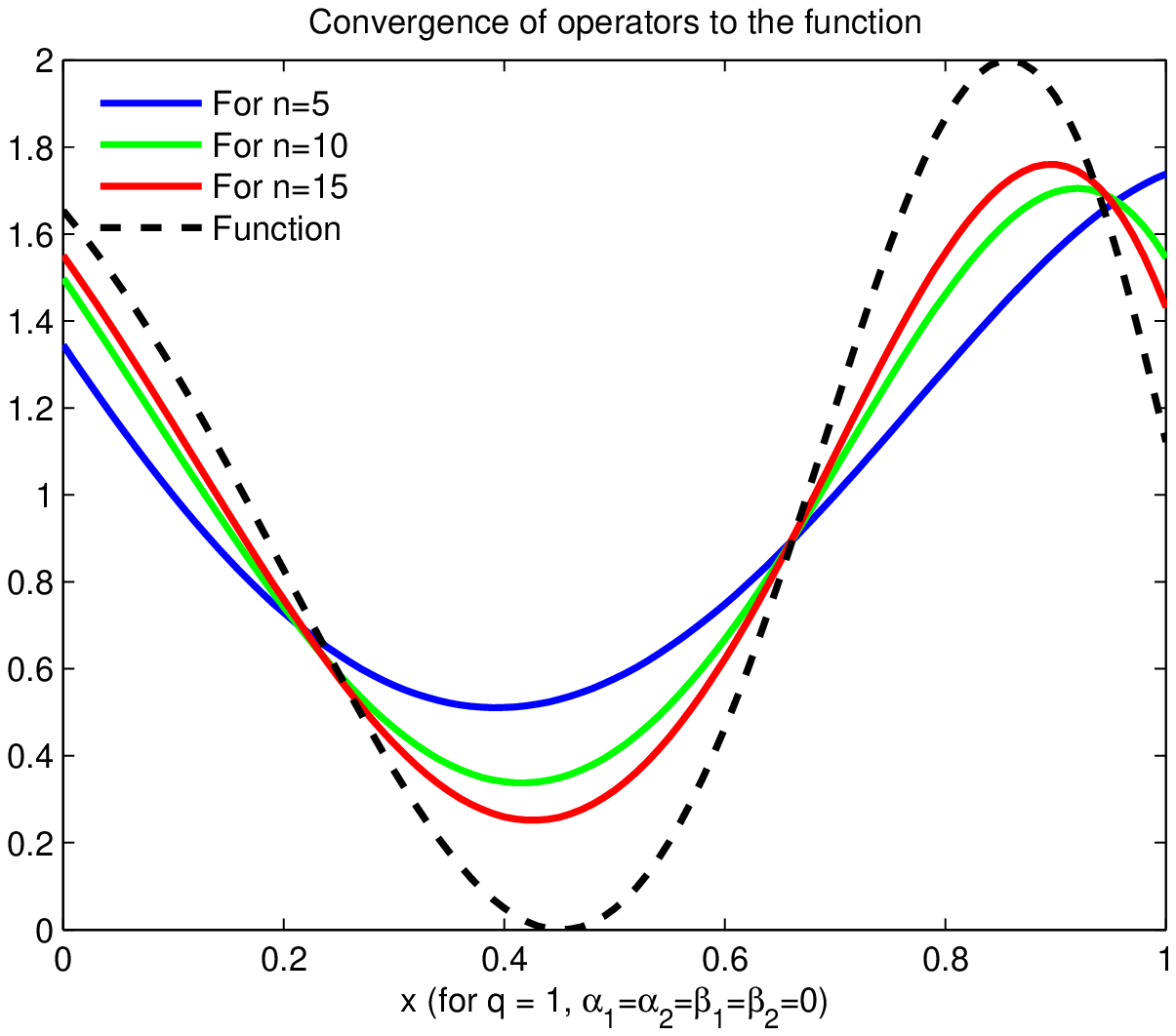}
\end{center}
\caption{}
\end{figure*}

\newpage

\begin{figure*}[htb!]
\begin{center}
\includegraphics[height=6cm, width=10cm]{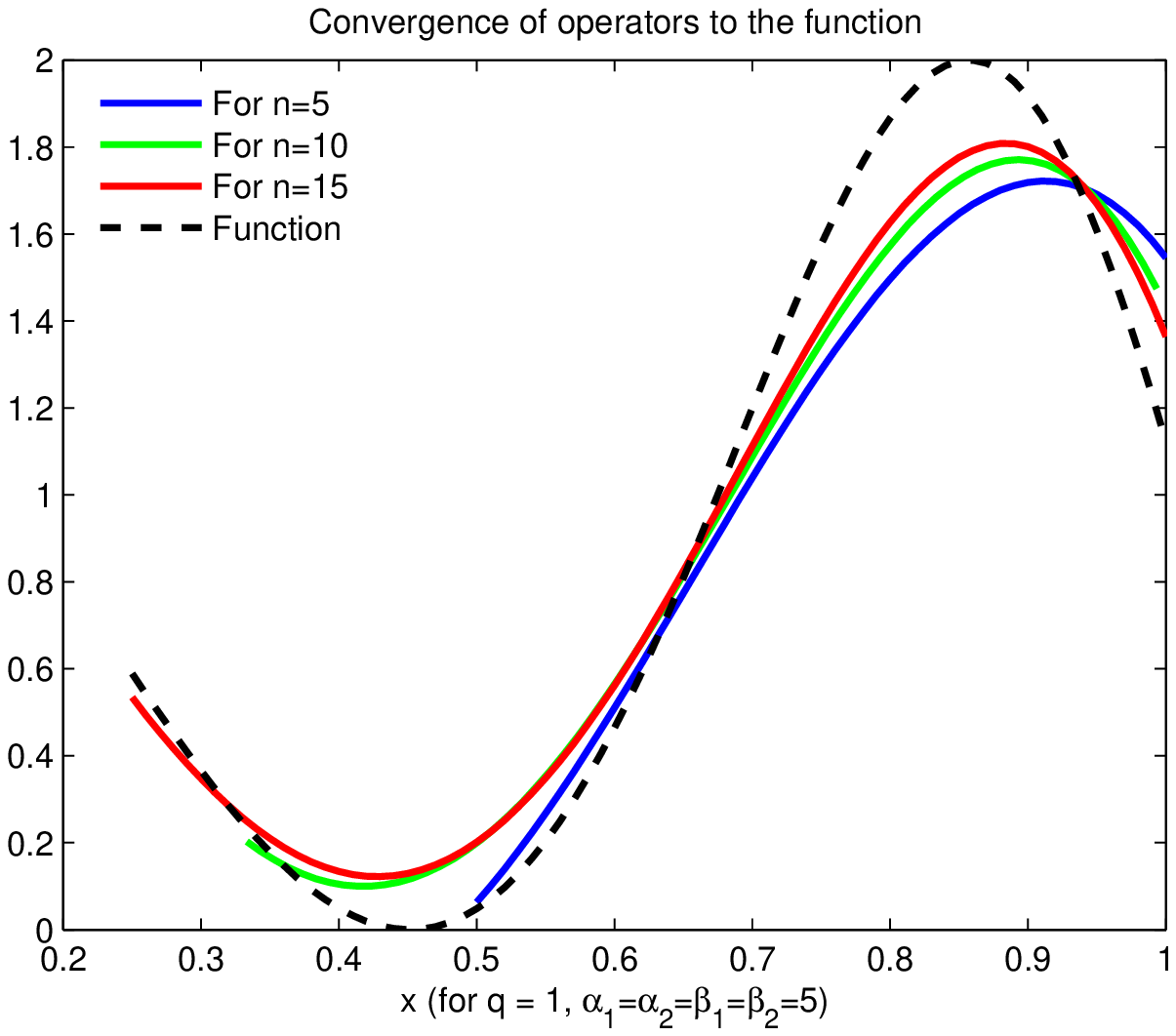}
\end{center}
\caption{}
\end{figure*}

\begin{figure*}[htb!]
\begin{center}
\includegraphics[height=6cm, width=10cm]{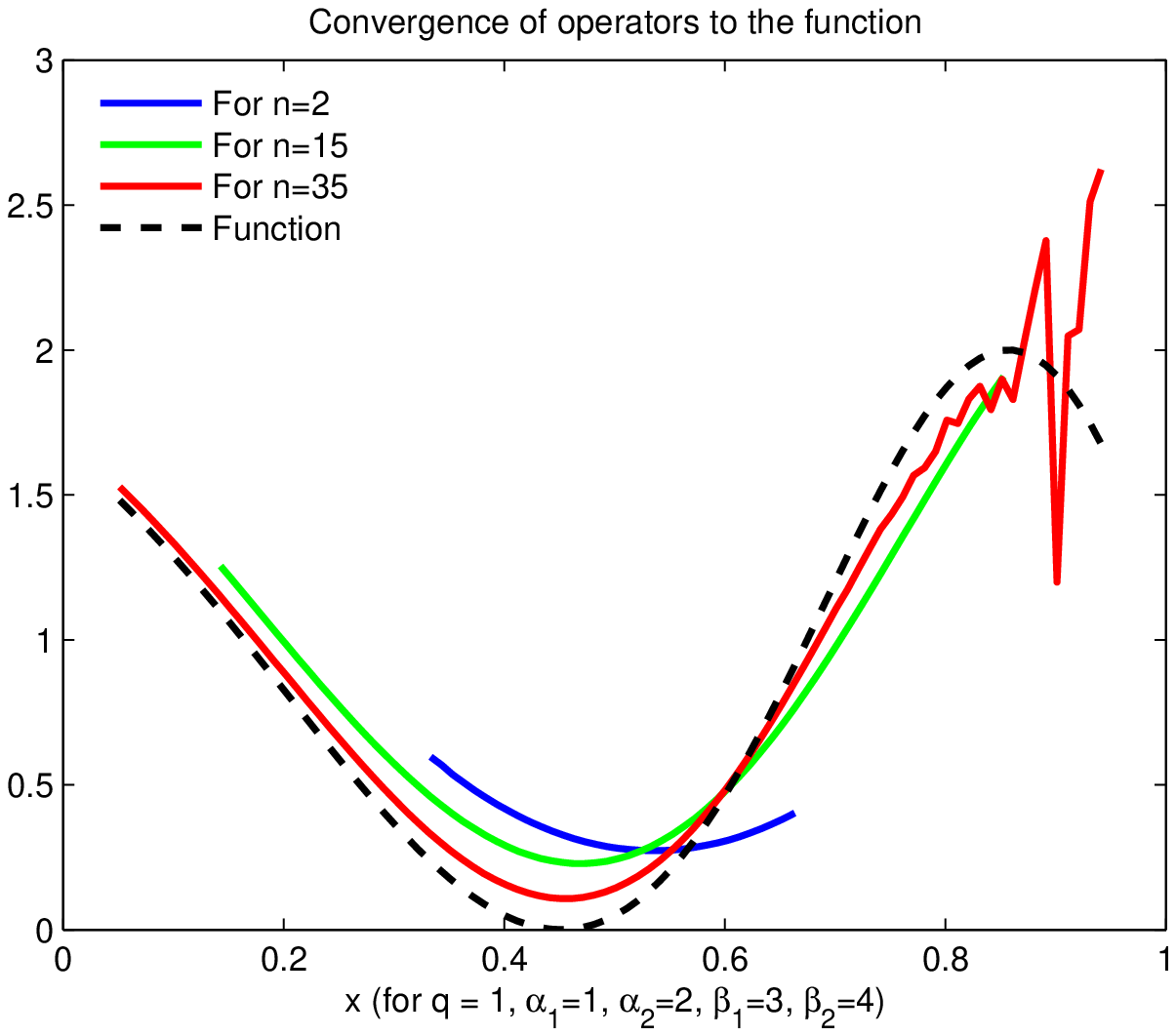}
\end{center}
\caption{}
\end{figure*}
Similarly as the value the $q$ increases, convergence of operators to the
function is shown in figure $4.4$ with different values of parameters $%
\alpha_1,~\alpha_2,~\beta_1,~\beta_2,~ \text{and}~ n$

\begin{figure*}[htb!]
\begin{center}
\includegraphics[height=6cm, width=10cm]{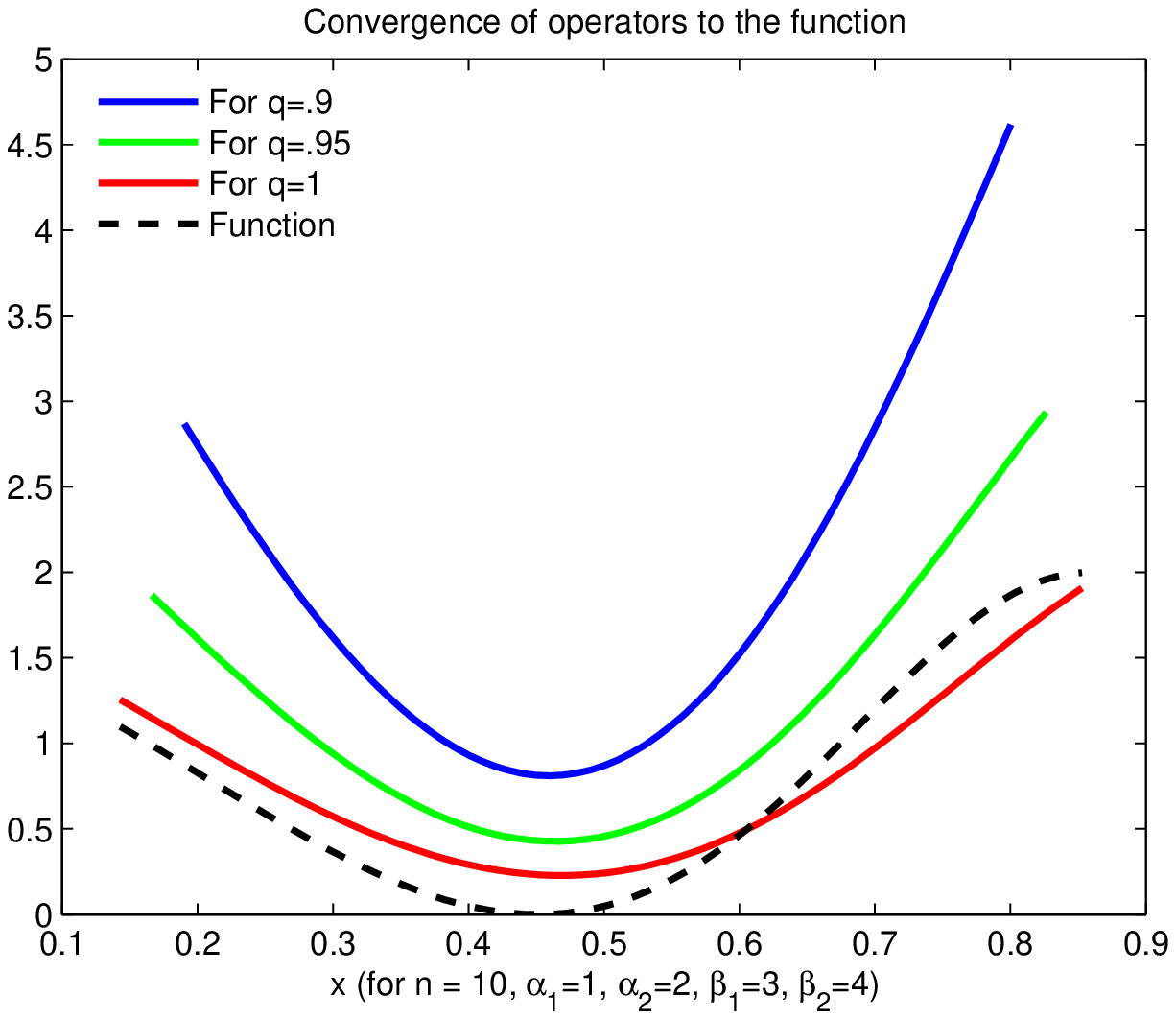}
\end{center}
\caption{}
\end{figure*}

\end{document}